\documentclass[a4paper]{article}

\usepackage{amsmath,amsfonts,amssymb,amsthm}
\usepackage{graphicx}
\usepackage{url}
\usepackage[latin1]{inputenc}

\title{{\large \textbf{Amenable actions of amalgamated free products of free groups over a cyclic subgroup and generic property}}}

\author{Soyoung Moon\footnote{This work is supported by Swiss NSF grant $20\_118014/1$.}}
\date{\today}

\newcommand{\Fix}{\textrm{Fix}}
\newcommand{\F}{\mathbb{F}} \newcommand{\Z}{\mathbb{Z}} \newcommand{\N}{\mathbb{N}}

\theoremstyle{plain}
\newtheorem{thm}{Theorem}
\newtheorem{prop}[thm]{Proposition}

\newtheorem{lem}[thm]{Lemma}

\theoremstyle{definition}
\newtheorem{defn}{Definition}[section]
\newtheorem{rem}{Remark}[section]

\begin{document}

\maketitle

\smallskip

\begin{abstract}
We show that the class of amalgamated free products of two free groups over a cyclic subgroup admits amenable, faithful and transitive actions on infinite countable sets. This work generalizes the results on such actions for doubles of free group on 2 generators.
\end{abstract}

\bigskip

\pagestyle{headings}

\section{Introduction}

An action of a countable group $G$ on a set $X$ is \textit{amenable} if there exists a sequence $\{A_n\}_{n\geq 1}$ of finite non-empty subsets of $X$ such that for every $g\in G$, one has
$$ \lim_{n\rightarrow \infty}\frac{|A_n \vartriangle g\cdot A_n|}{|A_n|}=0. $$
Such a sequence is called a \textit{F\o lner sequence} for the action of $G$ on $X$. Thanks to a result of F\o lner \cite{Fol}, this definition is equivalent to the existence of a $G$-invariant mean on subsets of $X$.

\begin{defn}
We say that a countable group $G$ is in the class $\mathcal{A}$ if it admits an amenable, faithful and transitive action on an infinite countable set.
\end{defn}
The question of understanding which groups are contained in $\mathcal{A}$ was raised by von Neumann and recently studied in a few papers (\cite{Van}, \cite{GlMo}, \cite{GrNe}, \cite{Moon}). In this note we add the following:

\begin{thm}\label{CycPinch}
Let $n$, $m\geq 1$. Let $G=\F_{m+1}\ast_{\Z}\F_{n+1}$ be an amalgamated free product of two free groups over a cyclic subgroup such that the image of the generator of $\Z$ is cyclically reduced in both free groups. Then any finite index subgroup of $G$ is in $\mathcal{A}$.
\end{thm}

The methods used in this work are analogous to those used in \cite{Moon} to obtain the theorem \ref{CycPinch} in case of $m=n=1$.
The role of the generic permutation $\alpha$ in \cite{Moon} is now played by a $n$-tuple of permutations $(\alpha_1,...,\alpha_n)$ and, for a cyclically reduced word $c=c(\alpha_1,...,\alpha_n)$, we now prove genericity of the set of such $n$-tuples for which the permutation $c$ has infinitely many orbits of size $k\in\N$, and all orbits finite. This new result allows us to apply the method of \cite{Moon} in our new setting.
\bigskip

For $X$ an infinite countable set, recall that $Sym(X)$ with the topology of pointwise convergence is a Baire space, i.e. every intersection of countably many dense open subsets is dense in $Sym(X)$. So for every $n\geq 1$, the product space $(Sym(X))^n$ is a Baire space. A subset of a Baire space is called \textit{meagre} if it is a union of countably many closed subsets with empty interior; and \textit{generic} or \textit{dense $G_{\delta}$} if its complement is meagre.

\begin{rem}
The amalgamated products appearing in Theorem \ref{CycPinch} are known in combinatorial group theory as ``\textit{cyclically pinched one-relator groups}'' (see \cite{Fine}). These are exactly the groups admitting a presentation of the form $G=\langle a_1, \dots, a_n, b_{1}, \dots, b_m | c=d\rangle$ where $1\neq c=c(a_1,\dots,a_n)$ is a cyclically reduced non-primitive word (not part of a basis) in the free group $\F_n=\langle a_1,\dots, a_{n}\rangle$, and $1\neq d=d(b_1,\dots,b_m)$ is a cyclically reduced non-primitive word in the free group $\F_m=\langle b_1,\dots,b_m\rangle$. The most important examples of such groups are the surface groups  i.e. the fundamental group of a compact surface. The fundamental group of the closed orientable surface of genus $g$ has the presentation
$\langle a_1, b_1, \dots, a_g, b_g | [a_1,b_1]\cdots [a_g,b_g]=1\rangle$.
By letting $c=[a_1,b_1]\cdots [a_{g-1},b_{g-1}]$ and $d=[a_g,b_g]^{-1}$, the group decomposes as the free product of the free group $\mathbb{F}_{2(g-1)}$ on $a_1, b_1, \dots, a_{g-1}, b_{g-1}$ and the free group $\mathbb{F}_2$ on $a_g, b_g$ amalgamated over the cyclic subgroup generated by $c$ in $\mathbb{F}_{2(g-1)}$ and $d$ in $\mathbb{F}_2$, hence it is a cyclically pinched one-relator group.
\end{rem}

\noindent \textbf{Acknowledgement.} The results presented here are part of my PhD thesis at Université de Neuchâtel, Switzerland. I would like to thank Alain Valette for his precious advice and constant encouragement and Yves Stalder for pointing out some mistakes in the previous version and numerous remarks.

\section{Graph extensions}

A graph $G$ consists of the set of vertices $V(G)$ and the set of edges $E(G)$, and two applications $E(G) \rightarrow E(G)$; $e \mapsto \bar{e}$ such that $\bar{\bar{e}}=e$ and $\bar{e}\neq e$, and $E(G)\rightarrow V(G)\times V(G)$; $e \mapsto (i(e), t(e))$ such that $i(e)=t(\bar{e})$. An element $e\in E(G)$ is a \textit{directed edge} of $G$ and $\bar{e}$ is the \textit{inverse edge} of $e$. For all $e\in E(G)$, $i(e)$ is the \textit{initial vertex} of $e$ and $t(e)$ is the \textit{terminal vertex} of $e$.

Let $S$ be a set. A \textit{labeling} of a graph $G=(V(G), E(G))$ on the set $S^{\pm 1}=S\cup S^{-1}$ is an application
$$
l:E(G) \rightarrow S^{\pm 1}; e \mapsto l(e)
$$
such that $l(\bar{e})=l(e)^{-1}$. A \textit{labeled graph} $G=(V(G), E(G),S,l)$ is a graph with a labeling $l$ on the set $S^{\pm 1}$. A labeled graph is \textit{well-labeled} if for any edges $e$, $e' \in E(G)$, $\big[ i(e)=i(e')$ and $l(e)=l(e')\big]$ implies that $e=e'$.

\bigskip

A word $w=w_m\cdots w_1$ on $\{ \alpha^{\pm 1}_n$, $\alpha^{\pm 1}_{n-1}$, $\dots$, $\alpha^{\pm 1}_1,\beta^{\pm 1} \}$ is called \textit{reduced} if $w_{k+1}\neq w_k^{-1}$, $\forall 1\leq k\leq m-1$. A word $w=w_m\cdots w_1$ on $\{ \alpha^{\pm 1}_n$, $\alpha^{\pm 1}_{n-1}$, $\dots$, $\alpha^{\pm 1}_1$, $\beta^{\pm 1} \}$ is called \textit{weakly cyclically reduced} if $w$ is reduced and $w_m\neq w_1^{-1}$; this definition allows $w_m$ and $w_1$ to be equal. Given a reduced word, we define two finite graphs labeled on $\{ \alpha^{\pm 1}_n$, $\alpha^{\pm 1}_{n-1}$, $\dots$, $\alpha^{\pm 1}_1,\beta^{\pm 1} \}$ as follows:

\begin{defn}\label{def-path}
Let $w=w_m\cdots w_1$ be a reduced word on $\{\alpha^{\pm 1}_k$, $\alpha^{\pm 1}_{k-1}$, $\dots$, $\alpha^{\pm 1}_1$, $\beta^{\pm 1}\}$. The \emph{path} of $w$ is a finite labeled graph $P(w, v_0)$ labeled on $\{\alpha^{\pm 1}_k$, $\dots$, $\alpha^{\pm 1}_1$, $\beta\}$ consisting of $m+1$ vertices and $m$ directed edges $\{e_1$, $\dots$, $e_m\}$ such that
\begin{enumerate}
\item[$\cdot$] $i(e_{j+1})=t(e_j)$, $\forall 1 \leq j \leq m-1$;
\item[$\cdot$] $v_0=i(e_1)\neq t(e_m)$;
\item[$\cdot$] $l(e_j)=w_j$, $\forall 1\leq j \leq m$.
\end{enumerate}
\end{defn}
The point $v_0$ is the \textit{startpoint} and the point $t(e_m)$ is the \textit{endpoint} of the path $P(w,v_0)$. The two points are the \textit{extreme points} of the path.

\begin{defn}\label{def-cycle}
Let $w=w_m\cdots w_1$ be a reduced word on $\{\alpha^{\pm 1}_k$, $\alpha^{\pm 1}_{k-1}$, $\dots$, $\alpha^{\pm 1}_1$, $\beta^{\pm 1}\}$. The \emph{cycle} of $w$ is a finite labeled graph $C(w, v_0)$ labeled on $\{\alpha^{\pm 1}_k$, $\dots$, $\alpha^{\pm 1}_1$, $\beta\}$ consisting of $m$ vertices and $m$ directed edges $\{e_1$, $\dots$, $e_m\}$ such that
\begin{enumerate}
\item[$\cdot$] $i(e_{j+1})=t(e_j)$, $\forall 1 \leq j \leq m-1$;
\item[$\cdot$] $v_0=i(e_1)= t(e_m)$;
\item[$\cdot$] $l(e_j)=w_j$, $\forall 1\leq j \leq m$.
\end{enumerate}
\end{defn}

The point $v_0$ is the \textit{startpoint} of the cycle $C(w,v_0)$.

Notice that since $w$ is a reduced word, the graph $P(w, v_0)$ is well-labeled. If $w$ is weakly cyclically reduced, then $C(w, v_0)$ is also well-labeled.

Conversely, if $P=\{e_1$, $e_2$, $\dots$, $e_n\}$ is a well-labeled path with $i(e_1)=v_0$, labeled by $l(e_i)=g_i$, $\forall i$, then there exists a unique reduced word $w=g_n\cdots g_1$ such that $P(w,v_0) $ is $P$. If $C=\{e_1$, $e_2$, $\dots$, $e_n\}$ is a well-labeled cycle with $t(e_n)=i(e_1)=v_0$, labeled by $l(e_i)=g_i$, $\forall i$, then there exists a unique weakly cyclically reduced word $w_1=g_n\cdots g_1$ such that $C(w,v_0)$ is $C$.

\bigskip

Let $X$ be an infinite countable set. Let $\beta$ be a simply transitive permutation of $X$. The \textit{pre-graph} $G_0$ is a labeled graph consisting of the set of vertices $V(G_0)=X$ and the set of directed edges all labeled by $\beta$ such that every vertex has exactly one entering edge and one outgoing edge, and $t(e)=\beta(i(e))$. One can imagine $G_0$ as the Cayley graph of $\mathbb{Z}$ with $1$ as a generator.

\begin{defn}
An \textit{extension} of $G_0$ is a well-labeled graph $G$ labeled by $\{\alpha^{\pm 1}_k$, $\alpha^{\pm 1}_{k-1}$, $\dots$, $\alpha^{\pm 1}_1$, $\beta^{\pm 1}\}$, containing $G_0$, with $V(G)=V(G_0)=X$. We will denote it by $G_0 \subset G$.
\end{defn}

In order to have a transitive action with some additional properties of the $\langle \alpha_k,\dots, \alpha_1, \beta \rangle$-action on $X$, we shall extend inductively $G_0$ on $1\leq i\leq k $ by adding finitely many directed edges labeled by $\alpha_i$ on $G_0$ where the edges labeled by $\beta$ are already prescribed. In order that the added edges represent an action on $X$, we put the edges in such a way that the extended graph is well-labeled, and moreover we put an additional edge labeled by $\alpha_i$ on every endpoint of the extended edges by $\alpha_i$; more precisely, if we have added $n$ edges labeled by $\alpha_i$ between $x_0$, $x_1$, $\dots$, $x_n$ successively, we put an $\alpha_i$-edge from $x_n$ to $x_0$ to have a cycle consisting of $n+1$ edges, which corresponds to a $\alpha_i$-orbit of size $n+1$. On the points where no $\alpha_i$-edges are involved, we can put any $\alpha_i$-edge in a way that the the extended graph is well-labeled and every point has a entering edge and a outgoing edge labeled by $\alpha_i$ (for example we can put a loop labeled by $\alpha_i$, corresponding to the fixed points). In the end, the graph represents an $\langle \alpha_k$, $\dots$, $\alpha_1$, $\beta \rangle$-action on $X$, i.e. $G$ will be a Schreier graph.

\begin{defn}
Let $G$, $G'$ be graphs labeled on a set $S^{\pm 1}$. A \textit{homomorphism} $f: G \rightarrow G'$ is a map sending vertices to vertices, edges to edges, such that
\begin{enumerate}
\item[$\cdot$] $f(i(e))=i(f(e))$ and $f(t(e))=t(f(e))$;
\item[$\cdot$] $l(e)=l(f(e))$,
\end{enumerate}
for all $e\in E(G)$.
\end{defn}
If there exists an injective homomorphism $f:G \rightarrow G'$, we say that $f$ is an \textit{embedding}, and $G$ \textit{embeds} in $G'$.

\begin{lem}\label{pathGen}
Let $k\geq 1$. Let $w_k=w_k(\alpha_k, \alpha_{k-1},\dots, \alpha_1,\beta)$ be a reduced word on $\{\alpha^{\pm 1}_k$, $\alpha^{\pm 1}_{k-1}$, $\dots$, $\alpha^{\pm 1}_1$, $\beta^{\pm 1}\}$. For every finite subset $F$ of $G_0$, there is an extension $G$ of $G_0$ on which the path $P(w_k,v_0)$ embeds in $G$, the image of $P(w_k,v_0)$ in $G$ does not intersect with $F$, and $G\setminus G_0$ is finite.
\end{lem}

\begin{proof}
Let us show this by induction on $k$. If $k=1$, it follows from Proposition 6 in \cite{Moon}. Indeed, in the proof of Proposition 6 in \cite{Moon}, we start by choosing any element $z_0\in X$ to construct a path. Since the set $X$ is infinite and there is no assumption on the starting point $z_0$ of the path, there are infinitely many choices for $z_0$.

For the proof of the induction step, consider the case
$$
w_k=\alpha^{a_{2m}}_k w^{2m-1}_{k-1}\alpha^{a_{2m-2}}_k\cdots \alpha^{a_4}_k w^3_{k-1}\alpha^{a_2}_{k} w^1_{k-1}.
$$
with $w^{i}_{k-1}=w^{i}_{k-1}(\alpha_{k-1},\dots, \alpha_1,\beta)$ a reduced word on $\{\alpha^{\pm 1}_{k-1}$, $\dots$, $\alpha^{\pm 1}_1$, $\beta^{\pm 1}\}$, for all $i$. To simplify the notation, we assume that $a_j$ is positive, $\forall j$.

Let $F\subset X$ be a finite subset of $X$. By hypothesis of induction, there is an extension $G_1$ of $G_0$ and an embedding $f^{1}$ such that $f^{1}:P(w^{1}_{k-1},v_0)\hookrightarrow G_1$ and the image of $P(w^1_{k-1},v_0)$ in $G_1$ does not intersect with $F$. Let $$f^{1}(v_0)=f^1\big(i(P(w^1_{k-1},v_0))\big)=:z_0$$ and $$f^1\big(t(P(w^1_{k-1},v_0))\big)=:z_1.$$ Inductively on each $2\leq i\leq m$, we apply the following algorithm:

\bigskip

\noindent \textbf{Algorithm}

\begin{enumerate}
\item Take an extension $G_{2i-2}$ of $G_0$ such that

\begin{enumerate}
\item[$\cdot$] $P(w^{2i-1}_{k-1}, v_{2i-2})$ embeds in $G_{2i-2}$ such that the image of $P(w^{2i-1}_{k-1}, v_{2i-2})$ does not intersect with $F$;
\item[$\cdot$] $ G_{2i-2}\cap G_{2i-3}=G_0$ (this is possible since there are infinitely many extensions $G'_{2i-2}$ of $G_0 $ by hypothesis of induction and $G_{2i-3}\setminus G_0$ is finite).
\end{enumerate}

\item Let $f^{2i-1}:P(w^{2i-1}_{k-1}, v_{2i-2})\hookrightarrow G_{2i-2}\cup G_{2i-3}=:G'_{2i-1}$ with
\begin{enumerate}
\item[$\cdot$] $f^{2i-1}\big(i( P(w^{2i-1}_{k-1}, v_{2i-2})  )\big)=f^{2i-1}(v_{2i-2})=:z_{2i-2}$;
\item[$\cdot$] $f^{2i-1}\big(t( P(w^{2i-1}_{k-1}, v_{2i-2})  )\big)=:z_{2i-1}$.
\end{enumerate}

\item Choose $|a_{2i-2}|-1$ points $\{ p^{(a_{2i-2})}_1$, $\dots$, $p^{(a_{2i-2})}_{|a_{2i-2}|-1}\}$ outside of the finite set of all points appeared until now, and put the directed edges labeled by $\alpha_k$ from
\begin{enumerate}
\item[$\cdot$] $z_{2i-3}$ to $p^{(a_{2i-2})}_1$;
\item[$\cdot$] $p^{(a_{2i-2})}_j$ to $p^{(a_{2i-2})}_{j+1}$, $\forall 1 \leq j \leq |a_{2i-2}|-2$;
\item[$\cdot$] $p^{(a_{2i-2})}_{|a_{2i-2}|-1}$ to $z_{2i-2}$,
\end{enumerate}
and let $G_{2i-1}:=G'_{2i-1}\cup\{$the additional $\alpha_k$-edges between $z_{2i-3}$ and $z_{2i-2}\}$.
\end{enumerate}

In the ends, we choose new $|a_{2m}|$ points $\{ p^{(a_{2m})}_1$, $\dots$, $p^{(a_{2m})}_{|a_{2m}|}\}$ and put the directed edges labeled by $\alpha_k$ from $z_{2m-1}$ to $p^{(a_{2m})}_1$, and from $p^{(a_{2m})}_j$ to $p^{(a_{2m})}_{j+1}$, $\forall 1 \leq j \leq |a_{2m}|$, so that we have $\alpha^{a_{2m}}_k z_{2m-1}=z_{2m}$.

By construction, the resulting graph $G_{2m-1}\cup P(\alpha^{a_{2m}}, v_{2m-1})=:G$ is an extension of $G_0$ satisfying $P(w_k,v_0)\hookrightarrow G$ such that the image of $P(w_k,v_0)$ does not intersect with $F$.
\end{proof}

\begin{lem}\label{cycleGen}
Let $w=w(\alpha_{n},\dots, \alpha_1,\beta)$ be a weakly cyclically reduced word on $\{\alpha^{\pm 1}_{n}$, $\dots$, $\alpha^{\pm 1}_1$, $\beta^{\pm 1}\}$ such that $\alpha_{i}$ appears in the word $w$ for some $i$ (i.e. $w\notin \langle\beta\rangle$). For every finite subset $F$ of $G_0$, there exists an extension $G_{n+1}$ of $G_0$ such that the cycle $C(w,v_0)$ embeds in $G_{n+1}$ and the image of $C(w,v_0)$ in $G_0$ does not intersect with $F$.
\end{lem}

\begin{proof}
Let us consider the case
$$
w=\alpha^{a_{2m}}_{i} w_{2m-1}\alpha^{a_{2m-2}}_{i}\cdots \alpha^{a_4}_{i} w_3\alpha^{a_2}_{i} w_1
$$
written as the normal form of $\langle \alpha_n, \dots, \alpha_{i+1}, \alpha_{i-1},\dots,\alpha_1,\beta\rangle \ast \langle \alpha_i\rangle$.

Since $w'=w_{2m-1}\alpha^{a_{2m-2}}_{i}\cdots \alpha^{a_4}_{i} w_3\alpha^{a_2}_{i} w_1$ is reduced, by Lemma \ref{pathGen}, there is an extension $G'_{n+1}$ of $G_0$ and a homomorphism $f:P(w', v_0)\rightarrow G'_{n+1}$ such that $f(P(w', v_0))$ is a path in $G'_{n+1}$ outside of $F$. Let $f(v_0)=:z_0$ be the startpoint of $f(P(w',v_0))$ and $f(w'(z_0))=:z_{2m-1}$ be the endpoint of $f(P(w',v_0))$.

 Choose $|a_{2m}|-1$ new points $\{ p_{a_m}$, $\dots$, $p_{|a_{2m}|-1}\}$ and put the directed edges labeled by $\alpha^{sign(a_{2m})}_i$ from
\begin{enumerate}
\item[$\cdot$] $z_{2m-1}$ to $p_1$;
\item[$\cdot$] $p_j$ to $p_{j+1}$, $\forall 1 \leq j \leq |a_{2m}|-2$;
\item[$\cdot$] $p_{|a_{2m}|-1}$ to $z_{0}$.
\end{enumerate}

By construction, the resulting graph $G_{n+1}:=G'_{n+1}\cup P(\alpha^{a_{2m}}, v_{2m-1})$ is an extension of $G_0$ and $C(w,v_0)$ embeds in $G_{n+1}$ outside of $F$.
\end{proof}

Let $c=c(\alpha_{n},\dots,\alpha_1,\beta)$ be a weakly cyclically reduced word on $\{\alpha^{\pm 1}_{n}$, $\dots$, $\alpha^{\pm 1}_1$, $\beta^{\pm 1}\}$ such that $c\notin \langle \beta \rangle$ and $w=w(\alpha_{n}$, $\alpha_{n-1}$, $\dots$, $\alpha_1$, $\beta)$ be a reduced word on $\{\alpha^{\pm 1}_{n}$, $\dots$, $\alpha^{\pm 1}_1$, $\beta^{\pm 1}\}$ such that $w\notin \langle c \rangle$.
Let $C(c, v_0)$ be the cycle of $c$ with startpoint at $v_0$, and let $P(w, v_0)$ be the path of $w$ with the same startpoint $v_0$ as $C(c,v_0)$ such that every vertex of $P(w,v_0)$ (other than $v_0$) is distinct from every vertex in $C(c,v_0)$. Let $C(c, wv_0)$ be the cycle of $c$ with startpoint at $wv_0$ such that every vertex of $C(c,wv_0)$ (other than $wv_0$) is distinct from every vertex in $P(w, v_0)\cup C(c, v_0)$. Let us denote by $Q_0(c,w)$ the union of $C(c, v_0)$, $P(w, v_0)$ and $C(c, wv_0)$. Let $Q(c,w)$ be the well-labeled graph obtained from $Q_0(c,w)$ by identifying the successive edges with the same initial vertex and the same label.
Notice that the well-labeled graph $Q(c,w)$ can have one, two or three cycles, and in each type of $Q(c,w)$, the quotient map $Q_0(c,w) \twoheadrightarrow Q(c,w)$ restricted to $C(c,v_0)$ and to $C(c,wv_0)$ is injective (each one separately).

\begin{lem}\label{FF}
There is an extension $G_{n+1}$ of $G_0$ such that $Q(c,w)$ embeds in $G_{n+1}$.
\end{lem}

\begin{figure}
\centering
\includegraphics[width=4cm]{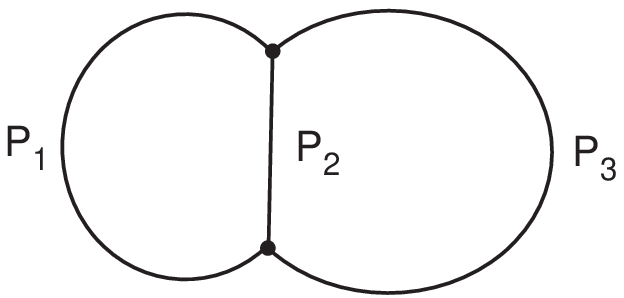}
\caption{}
\label{type3}
\end{figure}

\begin{proof}
By Lemma \ref{pathGen} and \ref{cycleGen}, it is enough to show that every cycle in $Q$ contains edges labeled by $\alpha^{\pm 1}_{i}$ for some $i$. For the cases where $Q$ has one or two cycles, it is clear since the cycles in $Q$ represent $C(c,v_0)$ and $C(c,wv_0)$, and $c\notin \langle \beta \rangle$. In the case where $Q(c,w)$ has three cycles, $Q(c,w)$ has three paths $P_1$, $P_2$ and $P_3$ such that $P_1\cap P_2\cap P_3$ are exactly two extreme points of $P_i$'s, and $P_1\cup P_2$, $P_2\cup P_3$ and $P_1\cup P_3$ are the three cycles in $Q(c,w)$ (see Figure \ref{type3}). So we need to prove that, if one of the three paths has edges labeled only on $\{\beta^{\pm 1}\}$, then the other two paths both contains edges labeled by $\alpha^{\pm 1}_{i}$ for some $i$. For this, it is enough to prove:

\bigskip

\noindent \textbf{Claim.} If the reduced word $c=\gamma\lambda$ is conjugate to the reduced word $\gamma\lambda'$ via a reduced word $w$, where $\gamma\in \langle \alpha_n$, $\alpha_{n-1}$, $\dots$, $\beta \rangle \setminus \langle\beta\rangle$ and $\lambda\in\langle\beta\rangle$, then $wc=cw$. Furthermore, the word $c$ can not be conjugate to the reduced word $\gamma^{-1}\lambda'$ with $\lambda'\in \langle \beta\rangle$.

\bigskip

Let us see how we can conclude Lemma \ref{FF} using the Claim. First of all, notice that $c$ does not commute with $w$ since we are treating the case where $Q$ has three cycles. More precisely, in a free group, two elements commute if and only if they are both powers of the same word. So if $cw=wc$, then $c=\gamma^k$ and $w=\gamma^l$ with $k\neq l$, where $\gamma$ is a non-trivial word, so that $Q$ has one cycle. Suppose that $P_1$ consists of edges labeled only on $\{\beta^{\pm 1}\}$. One of the cycles among $P_1\cup P_2$, $P_2\cup P_3$ and $P_1\cup P_3$ consists of edges labeled by the letters of $c$ up to cyclic permutation, let us say $P_1\cup P_2$ (i.e. if $c=c_1\cdots c_m$, given any startpoint $v_0$ in $P_1\cup P_2$, the directed edges of the cycle $C(c,v_0)$ are labeled on a cyclic permutation of the sequence $\{c_m$, $\dots$, $c_1\}$). Another cycle among $P_2\cup P_3$ and $P_1\cup P_3$ consists of edges labeled by the letters of the reduced form of $w^{-1}cw$ up to cyclic permutation. Since $c\notin \langle\beta\rangle$, the path $P_2$ has edges labeled by $\alpha^{\pm 1}_i$ for some $i$. Now, if the cycle representing $w^{-1}cw$ is $P_1\cup P_3$, then the path $P_3$ has edges labeled by $\alpha^{\pm 1}_i$ since $w^{-1}cw\notin \langle\beta\rangle$ and $P_1$ has only edges labeled on $\{\beta^{\pm 1}\}$(this is because two words in the free group $\mathbb{F}$ define conjugate elements of $\mathbb{F}$ if and only if their cyclic reduction in $\mathbb{F}$ are cyclic permutations of one another). Suppose now that the cycle representing $w^{-1}cw$ is $P_2\cup P_3$ and $P_3$ has edges labeled only on $\{\beta^{\pm 1}\}$. Then, $c$ would be the form $\gamma\lambda$ up to cyclic permutation where $\gamma\in \langle \alpha_n$, $\alpha_{n-1}$, $\dots$, $\beta \rangle \setminus \langle\beta\rangle$ (representing $P_2$) and $\lambda\in \langle\beta\rangle$ (representing $P_1$); and $w^{-1}cw$ would be the form $\gamma^{\pm 1}\lambda'$ up to cyclic permutation where $\lambda'\in \mathbb{F}_n$(representing $P_3$); but the Claim tells us that this is not possible, therefore $P_3$ contains edges labeled by $\alpha^{\pm 1}_i$ for some $i$.

\bigskip

Now we prove the Claim. Let $c=\gamma\lambda$ and $w^{-1}cw=\gamma\lambda'$ such that $\gamma\in \langle \alpha_n$, $\alpha_{n-1}$, $\dots$, $\beta \rangle \setminus \langle\beta\rangle$ and $\lambda$, $\lambda' \in \langle\beta\rangle$.
Without loss of generality, we can suppose that $\gamma=\gamma_m\lambda_{m-1}\cdots \lambda_1\gamma_1$, with $\gamma_i \in \langle \alpha_n$, $\alpha_{n-1}$, $\dots$, $\beta \rangle \setminus \langle\beta\rangle$ and $\lambda_i\in\langle\beta\rangle$. Since $\gamma\lambda$ and $\gamma\lambda'$ are conjugate in a free group, there exists $1\leq k\leq m$ such that
$$\gamma_k\lambda_{k-1}\cdots\lambda_1\gamma_1\lambda\gamma_m\lambda_{m-1}\cdots\gamma_{k+1}\lambda_k = \gamma\lambda'= \gamma_m\lambda_{m-1}\cdots \lambda_1\gamma_1\lambda'.$$
By identification of each letter, one deduces that $\lambda'=\lambda_k=\lambda_j$, for every $j$ multiple of $k$ in $\mathbb{Z}/m\mathbb{Z}$, and $\lambda=\lambda_{m-k}$. In particular, $\lambda=\lambda'$ so that $c=\gamma\lambda=\gamma\lambda'=w^{-1}cw$ and thus $cw=wc$. For the seconde statement, suppose by contradiction that there exists $w$ such that $w^{-1}cw=\gamma^{-1}\lambda'$. Then by the similar identification as above we deduce that $\lambda^{-1}=\lambda'$, so $w^{-1}cw$ would be a cyclic permutation of $c^{-1}$, which is clearly not possible.

\end{proof}

\section{Construction of generic actions of free groups}

Let $X$ be an infinite countable set. We identify $X=\mathbb{Z}$. Let $\beta$ be a simply transitive permutation of $X$ (which is identified to the translation $x\mapsto x+1$).

Let $c$ be a cyclically reduced word on  $\{ \alpha^{\pm 1}_n$, $\alpha^{\pm 1}_{n-1}$, $\dots$, $\alpha^{\pm 1}_1$, $\beta^{\pm 1} \}$ such that the sum $S_c(\beta)$ of the exponents of $\beta$ in the word $c$ is zero. Thus necessarily $c$ contains $\alpha_i$ for some $i$.

Let us denote by $S^{+}_c(\beta)$ the sum of positive exponents of $\beta$ in the word $c$; by denoting $S^{-}_c(\beta)$ the sum of negative exponents of $\beta$ in the word $c$, we have $0=S_c(\beta)=S^{+}_c(\beta)+S^{-}_c(\beta)$ (for example, if $c=\alpha_1 \beta^{-1}\alpha_2 \beta^{-1}\alpha^2_n \beta^2$, then $S^{+}_c(\beta)=2$). If $c$ does not contain $\beta$, we set $S^{+}_c(\beta)=0$.

Let $\{A_{m}\}_{m\geq 1}$ be a sequence of pairwise disjoint intervals of $X$ such that $|A_m|\geq m+ 2S^{+}_c(\beta)$, $\forall m\geq 1$. Clearly this sequence is a pairwise disjoint F\o lner sequence for $\beta$.

\begin{prop}\label{Fnew}
Let $c$ be a cyclically reduced word as above.
There exists $\alpha=(\alpha_1,\dots,\alpha_n) \in (Sym(X))^n$ such that $\langle \alpha_n$, $\alpha_{n-1}$, $\dots$, $\alpha_1$, $\beta \rangle$ is free of rank $n+1$, and
\begin{enumerate}
\item[(1)] the action of $\langle \alpha_n$, $\alpha_{n-1}$, $\dots$, $\alpha_1$, $\beta \rangle$ on $X$ is transitive and faithful;
\item[(2)] for all non trivial word $w$ on $\{\alpha^{\pm 1}_n$, $\alpha^{\pm 1}_{n-1}$, $\dots$, $\alpha^{\pm 1}_1$, $\beta^{\pm 1} \}$ with $w\notin\langle c\rangle$, there exist infinitely many $x\in X$ such that $cx=x$, $cwx=wx$ and $wx \neq x$;
\item[(3)] there exists a pairwise disjoint F\o lner sequence $\{A_k\}_{k\geq 1}$ for $\langle \alpha_n$, $\alpha_{n-1}$, $\dots$, $\alpha_1$, $\beta \rangle$ which is fixed by $c$, and $|A_k|=k$, $\forall k\geq 1$;
\item[(4)] for all $k\geq 1$, there are infinitely many $\langle c \rangle$-orbits of size $k$;
\item[(5)] every $\langle c \rangle$-orbit is finite;
\item[(6)] for every finite index subgroup $H$ of $\langle \alpha_n$, $\alpha_{n-1}$, $\dots$, $\alpha_1$, $\beta \rangle$, the $H$-action on $X$ is transitive.
\end{enumerate}
\end{prop}

With the notion of the permutation type, the conditions (4) and (5) mean that the word $c$ has the permutation type $(\infty$, $\infty$, $\dots$, $;$ $0)$.

\begin{proof}

For the proof, we are going to exhibit six generic subsets of $(Sym(X))^n$ that will do the job.

We start by claiming that the set
$$
\mathcal{U}_1 = \{\alpha=(\alpha_1,\dots,\alpha_n)\in (Sym(X))^n \mid \forall k \in \mathbb{Z}\setminus \{0\}, \exists x\in X \textrm{ such that } c^kx\neq x\}
$$
is generic in $(Sym(X))^n$.
Indeed, for every $k \in \mathbb{Z}\setminus \{0\}$, let $\mathcal{V}_k=\{\alpha\in (Sym(X))^n| \forall x\in X$, $c^kx= x \}$. The set $\mathcal{V}_k$ is closed since if $\{\gamma_m\}_{m\geq 1}$ is a sequence in $\mathcal{V}_k$ converging to $\gamma$, then $c^k(\gamma_m)$ converges to $c^k(\gamma)$. To see the interior of $\mathcal{V}_k$ is empty, let $\alpha\in \mathcal{V}_k$ and let $F\subset X$ be a finite subset. There is an extension $G_{n+1}$ of $G_0$ such that $P(c^k(\alpha'),v_0)$ embeds in $G_{n+1}$ outside of $F$ by Lemma \ref{pathGen}. So in particular there is $x\in X\setminus F$ such that $c^k(\alpha')x\neq x$, so $\alpha'\notin \mathcal{V}_k$. By defining $\alpha'|_F=\alpha|_F$, we have shown that $\mathcal{U}_1$ is generic in $(Sym(X))^n$.

\bigskip

Let us show that the set

\bigskip

\noindent $ \mathcal{U}_2= \{ \alpha=(\alpha_1,\dots,\alpha_n) \in (Sym(X))^n | \textrm{ for every } w\neq 1 \in \langle\alpha_n,\dots,\alpha_1,\beta\rangle \setminus \langle c \rangle,$

$\,$ there exist infinitely many $x\in X$ such that $cx=x$, $cwx=wx$ and $wx \neq x\}$

\bigskip

\noindent is generic in $(Sym(X))^n$.

Indeed, for every non trivial word $w$ in $\langle\alpha_n,\dots,\alpha_1,\beta\rangle \setminus \langle c \rangle$, let
$\mathcal{V}_w =\{\alpha\in (Sym(X))^n|$ there exists a finite subset $K\subset X$ such that $(\textrm{Fix}(c)\cap w^{-1}\textrm{Fix}(c)\cap \textrm{supp}(w))\subset K\}=\bigcup_{K \textrm{finite}\subset X}\{\alpha\in (Sym(X))^n|(\textrm{Fix}(c)\cap w^{-1}\textrm{Fix}(c)\cap \textrm{supp}(w))\subset K\}$.
We shall show that the set $\mathcal{V}_w$ is meagre.
It is an easy exercise to show that the set
$$\mathcal{V}_{w,K}=\{\alpha\in (Sym(X))^n|(\textrm{Fix}(c)\cap w^{-1}\textrm{Fix}(c)\cap \textrm{supp}(w))\subset K\}$$
is closed. To show that the interior of $\mathcal{V}_{w,K}$ is empty, let $\alpha\in \mathcal{V}_{w,K}$, and $F\subset X$ be a finite subset. We need to prove that for some $\alpha'$ defined as $\alpha'|_F=\alpha|_F$, we can extend the definition of $\alpha'$ outside of the finite subset such that $\alpha'\notin \mathcal{V}_{w,K}$. By Lemma \ref{FF}, we can take an extension $G_{n+1}$ of $G_0$ such that $Q(c(\alpha'),w)$ embeds in $G_{n+1}$ outside of $F\cup \alpha(F)\cup K$, which proves the genericity of $\mathcal{U}_2$.

\bigskip

Now let us show that the set

\bigskip

\noindent $\mathcal{U}_3 =\{ \alpha=(\alpha_1,\dots,\alpha_n) \in (Sym(X))^n |$ there exists $\{A_{m_k}\}_{k\geq 1}$ a subsequence of

$\,\,\,\,\, \{A_m\}_{m\geq 1}$ such that $A_{m_k}\subset \mathrm{Fix}(\alpha_i)$, $\forall k\geq 1$, $\forall 1\leq i\leq n$\}

\bigskip

\noindent is generic in $(Sym(X))^n$.

Indeed, the set $\mathcal{U}_3$ can be written as $\mathcal{U}_3=\bigcap_{N\geq 1}\{\alpha=(\alpha_1,\dots,\alpha_n)\in (Sym(X))^n| \exists k\geq N$ such that $A_k\subset \Fix(\alpha_i)$, $\forall i\}$. We claim that for every $N\geq 1$, the set $\mathcal{V}_N=\{\alpha\in (Sym(X))^n| \forall k\geq N$, $A_k \subsetneq \cap_i \Fix(\alpha_i)\}$ is closed and of empty interior. It is closed since $\mathcal{V}_N=\bigcap_{k\geq N}\{\alpha\in (Sym(X))^n|A_k \subsetneq \cap_i \Fix(\alpha_i)\}$ and the set $\{\alpha\in (Sym(X))^n|A_k \subsetneq \cap_i \Fix(\alpha_i)\}$ is clearly closed. For the emptiness of its interior, let $\alpha\in\mathcal{V}_N$ and let $F\subset X$ be a finite subset. Let $k\geq N$ such that $A_k\cap (F\cup \alpha(F))=\emptyset$. We can then take $\alpha'\in (Sym(X))^n$ fixing $A_k$ and satisfying $\alpha'|_F=\alpha|_F$.

\bigskip

For (4), we show that the set
\bigskip

\noindent $\mathcal{U}_4=\{\alpha=(\alpha_1,\dots,\alpha_n)\in (Sym(X))^n| \forall m$, there exist infinitely many $\langle c\rangle$-orbits

$\,\,\,\,\,$ of size $m$ \}

\bigskip
\noindent is generic in $(Sym(X))^n$.

For all $m\geq 1$, let
$\mathcal{V}_m=\{\alpha\in (Sym(X))^n|$ there exists a finite subset $K\subset X$ such that every $\langle c \rangle$-orbit of size $m$ is contained in $K\}$ $=$ $\bigcup_{K \textrm{ finite}\subset X}$ $\mathcal{V}_{m,K}$,
where
$$\mathcal{V}_{m,K}= \{\alpha\in (Sym(X))^n| \textrm { if } |\langle c \rangle \cdot x|=m, \textrm{ then } \langle c \rangle \cdot x \subset K\}.$$

\noindent \textit{$\cdot$ $\mathcal{V}_{m,K}$ is of empty interior.} $\quad$ Let $F\subset X$ be a finite subset. Let $\alpha\in \mathcal{V}_{m,K}$. Take $x\notin (F\cup \alpha (F))\cup K$. Since $c$ contains $\alpha_i$ for some $i$, we can construct a cycle $c^m(\alpha')$ outside of $F\cup \alpha (F)\cup K$ such that $\alpha'|_F=\alpha |_F$ (Lemma \ref{cycleGen}), so that the orbit of $x$ under $\alpha'$ is of size $m$ and not contained in $K$.

\noindent \textit{$\cdot$ $\mathcal{V}_{m,K}$ is closed.} $\quad$ Let $\{\gamma_l\}_{l\geq 1}\subset \mathcal{V}_{m,K}$ converging to $\gamma \in (Sym(X))^n$. Let $x\in X$ such that $|\langle c(\gamma)\rangle \cdot x|=m$. Since $\gamma_l$ converges to $\gamma$, $c(\gamma_l)$ converges to $c(\gamma)$. Since $\langle c(\gamma)\rangle \cdot x$ is finite, there exists $l_0$ such that $\langle c(\gamma)\rangle \cdot x = \langle c(\gamma_l)\rangle \cdot x$, $\forall l\geq l_0$. Since $\gamma_l\in \mathcal{V}_{m,K}$ and $m=|\langle c(\gamma)\rangle \cdot x|=|\langle c(\gamma_l)\rangle \cdot x|$, we have $\langle c(\gamma_l)\rangle \cdot x\subset K$, $\forall l\geq l_0$. Therefore $\langle c(\gamma)\rangle \cdot x\subset K$, so that $\gamma\in \mathcal{V}_{m,K}$.

\bigskip

About (5), we prove that the set
$$
\mathcal{U}_5=\{\alpha=(\alpha_1,\dots,\alpha_n)\in (Sym(X))^n| \textrm{$\forall x\in X$, $\langle c\rangle \cdot x$ is finite }\}
$$
is generic in $(Sym(X))^n$.

For all $x\in X$, let $\mathcal{V}_x=\{\alpha\in (Sym(X))^n| \langle c\rangle \cdot x \textrm{ is infinite }\}$.
It is clear that the set $\mathcal{V}_x$ is closed. To see that the interior of $\mathcal{V}_{x}$ is empty, let $F\subset X$ be a finite subset and let $\alpha\in \mathcal{V}_{x}$. We shall show that there exists $\alpha'\notin \mathcal{V}_x$ such that $\alpha|_F=\alpha'|_F$. Denote $c=c(\alpha)$ and $c'=c(\alpha')$. We choose $p>> 1$ large enough so that
\begin{displaymath}
\left\{ \begin{array}{l}
 \big(B(c^{-p-1}x,|c|)\cup B(c^{p+1}x,|c|)\big)\cap (F\cup \alpha(F))=\emptyset;\\
 (F\cup \alpha(F))\subset B(x,|c^p|),
 \end{array} \right.
\end{displaymath}
where $|c|$ is the length of $c$ and $B(x,r)$ is the ball centered on $x$ with the radius $r$.

We construct a path of $c'$ outside of $B(x,|c^p|)$ starting from $c^{p+1}x$ which ends on $c^{-p-1}x$, i.e. $c'(c^{p+1}x)=c^{-p-1}x$. This is possible since $c'$ contains $\alpha_i$ for some $i$ (Lemma \ref{pathGen}). On the points in $B(x,|c^{p+1}|)$, we define
$$
\alpha'|_{B(x,|c^{p+1}|)}=\alpha|_{B(x,|c^{p+1}|)}.
$$
In particular, $\alpha'|_F=\alpha|_{F}$, and $|\langle c'\rangle\cdot x|$ is finite.

\bigskip

Finally for (6), let

\bigskip

\noindent $\mathcal{U}_6=\{\alpha=(\alpha_n,\dots,\alpha_1)\in(Sym(X))^n |$ for every finite index subgroup $H$ of

$\,\,\,\,\,\,\, \langle \alpha_1,\beta\rangle$, the $H$-action on $X$ is transitive $\}$.

\bigskip

By Proposition 4 in \cite{Moon}, the set $\mathcal{W}=\{\alpha_1 \in Sym(X)|$ for every finite index subgroup $H$ of $\langle \alpha_1,\beta\rangle$, the $H$-action on $X$ is transitive $\}$ is generic in $Sym(X)$. Thus $\mathcal{U}_6$ is generic in $(Sym(X))^n$ since $\mathcal{U}_6=\mathcal{W}\times (Sym(X))^{n-1}$.

\bigskip

Now let $\alpha=(\alpha_1,\dots,\alpha_n)\in \cap_{i=1}^6 \mathcal{U}_i$. It remains us to prove (3) and (6) in the Proposition. To simplify the notation, let $A_m:=A_{m_k}$ be the subsequence of $A_m$ fixed by $\alpha_i$, $\forall 1\leq i\leq n$ (genericity of $\mathcal{U}_3$).

Without loss of generality, let $c=w_1\beta^{b_1}w_2\beta^{b_2}\cdots w_l\beta^{b_l}$, where $w_j$ are reduced words on $\{\alpha^{\pm 1}_n$, $\dots$, $\alpha^{\pm 1}_1\}$, $\forall 1\leq j\leq l$. Recall that $\{A_{m}\}_{m\geq 1}$ is a sequence of pairwise disjoint intervals such that $|A_m|\geq m+ 2S^{+}_c(\beta)$. If $c$ does not contain $\beta$, then we can take the subinterval $A'_m$ of $A_m$ such that $|A'_m|=m$ for the F\o lner sequence which is fixed by $c$. If not, for all $m> S^{+}_c(\beta)$, let
$$
E_m=\beta^{b_1}(A_m) \cap \beta^{b_2+b_1}(A_m)\cap \cdots \cap \beta^{b_{l-1}+b_{l-2}+\cdots +b_1}(A_m)\cap \beta^{b_{l}+b_{l-1}+\cdots +b_1}(A_m).
$$

Notice that $\beta^{b_{l}+b_{l-1}+\cdots +b_1}(A_m)=A_m$. We claim that the set $E_m$ is not empty.
Indeed, for every $1 \leq i\leq l$, the set
$$\beta^{b_i+b_{i-1}+ \cdots + b_1}(A_m) \cap \beta^{b_p+b_{p-1}+\cdots +b_1}(A_m)$$
is not empty, $\forall 1\leq p\leq i-1$ since $|b_i+b_{i-1}+\cdots +b_{p+1}|\leq S^{+}_c(\beta)< |A_m|$. Moreover, a family of intervals which meet pairwise, has non-empty intersection so that $E_m\neq \emptyset$.


In addition, let us show that $c$ fixes the elements of $E_m$. Let $x\in E_m$ and let $1\leq p\leq l-1$. There exists $a_{l-p+1}\in A_m$ such that $x=\beta^{b_{l-p}+b_{l-p-1}+\cdots + b_1}(a_{l-p+1})$. Then
\begin{eqnarray}
\beta^{b_{l-p+1}+\cdots +b_{l-1}+b_l}(x)&=&\beta^{b_l+b_{l-1}+\cdots + b_{l-p+1}}(x)\nonumber \\
&=&\beta^{b_l+b_{l-1}+\cdots + b_{l-p+1}}\cdot\beta^{b_{l-p}+b_{l-p-1}+\dots +b_1}(a_{l-p+1})\nonumber \\
&=&a_{l-p+1} \in A_m.\nonumber
\end{eqnarray}

Since $w_j$ fixes every element in $A_m$, and the element $\beta^{b_{l-p+1}+\cdots +b_{l-1}+b_l}(x)$ is in $A_m$ for every $1\leq p\leq l-1$, the word $c$ fixes $x$, $\forall x\in E_m$. Clearly the set $E_m$ is a F\o lner sequence for $\langle \alpha_n$, $\alpha_{n-1}$, $\dots$, $\alpha_1$, $\beta\rangle$.

Furthermore, we have
$$
A_m \cap \beta^{S^{+}_c(\beta)}A_m \cap \beta^{S^{-}_c(\beta)}A_m \subseteq E_m,
$$
and
$$
|A_m \cap \beta^{S^{+}_c(\beta)}A_m \cap \beta^{S^{-}_c(\beta)}A_m |=|A_m|-2S^{+}_c(\beta)\geq m.
$$

So $|E_m|\geq m$, and upon replacing $E_{m}$ by a subinterval $E'_m$ of $E_m$ such that $|E'_{m}|=m$, we can suppose that $|E_{m}|=m$, $\forall m\geq 1$. Thus the sequence $\{E_m\}_{m\geq 1}$ is a F\o lner sequence satisfying the condition in (3) in the Proposition \ref{Fnew}.

Furthermore, if $H$ is a finite index subgroup of $\langle \alpha_n,\dots,\alpha_1, \beta\rangle$, then $Q=H\cap \langle\alpha_1,\beta\rangle$ is a finite index subgroup of $\langle \alpha_1,\beta\rangle$, so by the genericity of $\mathcal{U}_6$ the $Q$-action is transitive and therefore the $H$-action on $X$ is transitive.

\end{proof}

\section{Construction of $\mathbb{F}_{n+1}\ast_{\mathbb{Z}} \mathbb{F}_{m+1}$-actions, $n,m\geq 1$}

Let $X$ be an infinite countable set.
Let $G=\langle \alpha_n$, $\alpha_{n-1}$, $\dots$, $\alpha_1$, $\beta \rangle$ $\curvearrowright X$ be the group action constructed as in Proposition \ref{Fnew} with the pairwise disjoint F\o lner sequence $\{A_k\}_{k\geq 1}$. For $m\geq 1$, let $d$ be a cyclically reduced word on $\{\alpha_m$, $\alpha_{m-1}$, $\dots$, $\alpha_1$, $\beta \}$ such that $S_d(\beta)=0$ and $d$ contains $\alpha_j$ for some $j$. Let $H=\langle \alpha_m$, $\alpha_{m-1}$, $\dots$, $\alpha_1$, $\beta \rangle$ $\curvearrowright X$ be the group action constructed as in Proposition \ref{Fnew} with the pairwise disjoint F\o lner sequence $\{B_k\}_{k\geq 1}$. Let $Z=\{\sigma \in Sym(X)| \sigma c = d\sigma\}$. By virtue of the points (4) and (5) of Proposition \ref{Fnew}, the set $Z$ is not empty. Let
$$ H^{\sigma}=\sigma^{-1}H\sigma= \langle \sigma^{-1}\alpha_m\sigma, \sigma^{-1}\alpha_{m-1}\sigma, \dots, \sigma^{-1}\alpha_1\sigma, \sigma^{-1}\beta\sigma \rangle.$$

For $\sigma\in Z$, consider the amalgamated free product $G\ast_{\langle c=d\rangle}H^{\sigma}$ of $G$ and $H^{\sigma}$ along $\langle c=d \rangle$. The action of $G\ast_{\langle c=d\rangle}H^{\sigma}$ on $X$ is given by $g\cdot x=gx$, and $h\cdot x=\sigma^{-1}h\sigma x$, $\forall g\in G$ and $\forall h\in H$.

Notice that the set $Z$ is closed in $Sym(X)$. In particular, $Z$ is a Baire space.

\begin{prop}\label{fidel}
The set
$$
\mathcal{O}_1=\{\sigma\in Z \mid \textrm{ the action of $G\ast_{\langle c=d \rangle}H^{\sigma}$ on $X$ is faithful } \}
$$
is generic in $Z$.
\end{prop}

\begin{proof}

For every non trivial word $w\in G\ast_{\langle c=d\rangle}H^{\sigma}$, let us show that the set
$$\mathcal{V}_w=\{\sigma\in Z | \forall x\in X, w^{\sigma}x=x\}$$
 is closed and of empty interior. It is obvious that the set $\mathcal{V}_w$ is closed. To prove that the set $\mathcal{V}_w$ is of empty interior, let us treat the case where $w=ag_n h_n\cdots g_1 h_1$ with $a\in\langle c \rangle$, $g_i\in G \setminus \langle c \rangle$, and $h_i\in H \setminus \langle d \rangle$, $n\geq 1$. The corresponding element of $Sym(X)$ given by the action is $w^{\sigma}=ag_n\sigma^{-1}h_n\sigma \cdots g_1\sigma^{-1}h_1\sigma$. Let $\sigma\in \mathcal{V}_w$. Let $F\subset X$ be a finite subset. We shall show that there exists $\sigma'\in Z \setminus \mathcal{V}_w$ such that $\sigma'|_F=\sigma |_F$. For all $g\in G \setminus \langle c \rangle $ and $h\in H\setminus \langle d \rangle$, let
$$
\widehat{g}=\{x\in X \mid cx=x, \, cgx=gx \textrm{ and $gx\neq x$ }\},
$$
$$
\widehat{h}=\{x\in X \mid dx=x, \, dhx=hx \textrm{ and $hx\neq x$ }\}.
$$
By (2) of Proposition \ref{Fnew}, these sets are infinite.

Choose any $x_0 \in \textrm{Fix}(c)\setminus (F\cup \sigma(F))$. By induction on $1\leq i\leq n$, we choose $x_{4i-3}\in \widehat{h_i}$ such that $x_{4i-3}$, $h_ix_{4i-3} \notin (F\cup \sigma(F))$ are new points. This is possible since $\widehat{h_i}$ is infinite. Then we define
$$\sigma' (x_{4i-4}):=x_{4i-3} \textrm{ and } \sigma' (\sigma^{-1} (x_{4i-3})):=\sigma (x_{4i-4}).$$
 We set $x_{4i-2}:= h_i x_{4i-3}$, which is different from $x_{4i-3}$ and which is fixed by $d$, by definition of $\widehat{h_i}$. We choose $x_{4i-1}\in \widehat{g_i}$ such that $x_{4i-1}$, $g_ix_{4i-1}\notin (F\cup\sigma(F))$ are again new points. This is again possible since $\widehat{g_i}$ is infinite. Then we define
$$\sigma' (x_{4i-1}):=x_{4i-2}\textrm{ and }\sigma'(\sigma^{-1} (x_{4i-2})):=\sigma (x_{4i-1}).$$
 We finally set $x_{4i}:= g_i x_{4i-1}$.
 Then every point $x$ on which $\sigma'$ is defined verifies $\sigma'c(x)=d\sigma'(x)$. Indeed,
\begin{enumerate}
\item[$\cdot$] $\sigma'c(x_{4i-4})=\sigma'(x_{4i-4})=x_{4i-3}=d(x_{4i-3})=d\sigma'(x_{4i-4})$ since $x_{4i-4}\in \textrm{Fix}(c)$ and $x_{4i-3}\in \mathrm{Fix}(d)$;
\item[$\cdot$] $\sigma'c(\sigma^{-1}(x_{4i-3}))$ $=\sigma'(\sigma^{-1}(x_{4i-3}))$ $=\sigma(x_{4i-4})$ $=d\sigma (x_{4i-4})$ $=d\sigma'(\sigma^{-1}(x_{4i-3}))$ since $\sigma^{-1}(x_{4i-3})\in \Fix(c)$ and $\sigma(x_{4i-4})\in \Fix(d)$ because $\sigma\in Z$;
\item[$\cdot$] $\sigma'c(x_{4i-1})=\sigma'(x_{4i-1})=x_{4i-2}=d(x_{4i-2})=d\sigma'(x_{4i-1})$ since $x_{4i-2}\in \textrm{Fix}(d)$ and $x_{4i-1}\in \mathrm{Fix}(c)$;
\item[$\cdot$] $\sigma'c(\sigma^{-1}(x_{4i-2}))$ $=\sigma'(\sigma^{-1}(x_{4i-2}))$ $=\sigma(x_{4i-1})$ $=d\sigma(x_{4i-1})$ $=d\sigma'(\sigma^{-1}(x_{4i-2}))$ since $\sigma^{-1}(x_{4i-2})\in \Fix(c)$ and $\sigma(x_{4i-1})\in \Fix(d)$ because $\sigma\in Z$.
\end{enumerate}

By construction, the $4n$ points defined by the subwords on the right of $w^{\sigma'}$ are all distinct. In particular, $w^{\sigma'}x_0=x_{4n}\neq x_0$. If $w=h\in H\setminus \{\mathrm{Id}\}$, choose $x_0\in \Fix(c)\setminus (F\cup \sigma(F))$, $x_1\in \hat{h}\setminus (F\cup \sigma(F)\cup \{x_0\})$, $x_2\in \Fix(c)\setminus (F\cup \sigma(F)\cup \{x_0, x_1\})$ and define $\sigma'(x_0)=x_1$, $\sigma'(x_2)=hx_1$, $\sigma'(\sigma^{-1}(x_1))=\sigma(x_0)$, $\sigma'(\sigma^{-1}(hx_1))=\sigma(x_2)$ so that $w^{\sigma'}x_0=x_2\neq x_0$.
At last, if $w=g\in G\setminus \{\mathrm{Id}\}$, then there exists $x\in X$ such that $gx\neq x$ since $G$ acts faithfully on $X$. For all other points, we define $\sigma'$ to be equal to $\sigma$. Therefore, $\sigma'$ constructed in this way is in $Z\setminus \mathcal{V}_w$ and $\sigma' |_F=\sigma |_F$.
\end{proof}

\begin{prop}\label{moy}
The set
$$
\mathcal{O}_2=\{\sigma \in Z | \textrm{ $\exists$ $\{k_l\}_{l\geq 1}$ a subsequence of $k$ such that $\sigma(A_{k_l})=B_{k_l}$, $\forall l\geq 1$ }\}
$$
is generic in $Z$.
\end{prop}

\begin{proof}

Let us write $\mathcal{O}_2=\bigcap_{N\in \mathbb{N}}\{\sigma\in Z| \textrm{ there exists $n\geq N$ such that } \sigma(A_n)=B_n\}$. We need to show that for all $N
\in \mathbb{N}$, the set $\mathcal{V}_N=\{\sigma\in Z | \forall n\geq N, \sigma(A_n)\neq B_n\}$ is closed and of empty interior.

\noindent \textit{$\cdot$ $\mathcal{V}_N$ is of empty interior.} $\quad$ Let $\sigma\in \mathcal{V}_N$. Let $F\subset X$ be a finite subset. Let $n\geq N$ large enough so that $A_n\cap (F\cup\sigma(F))=\emptyset$ and $B_n\cap (F\cup\sigma(F))=\emptyset$. This is possible since the sets $\{A_n\}$ (respectively the sets $\{B_n\}$) are pairwise disjoint. Let $A_n=\{a_1$, $\dots$, $a_n\}$ and $B_n=\{b_1$, $\dots$, $b_n\}$. We define $\sigma'(a_i)=b_i$ and $\sigma'(\sigma^{-1}(b_i))=\sigma(a_i)$, $\forall i$, which is well defined because $a_i\in \textrm{Fix}(c)$ and $b_i\in \textrm{Fix}(d)$. For all other points, we define $\sigma'$ to be equal to $\sigma$. Therefore, $\sigma'\in Z\setminus \mathcal{V}_N$ and $\sigma'|_F=\sigma|_F$.

\noindent \textit{$\cdot$ $\mathcal{V}_N$ is closed.} $\quad$ We have $\mathcal{V}_N=\bigcap_{n\geq N}\mathcal{W}_{n}$, where $\mathcal{W}_{n}=\{\sigma\in Z|\sigma(A_n)\neq B_n\}$. So the set $\mathcal{V}_N$ is closed being the intersection of closed sets.
\end{proof}

Let $\sigma \in \mathcal{O}_1\cap \mathcal{O}_2$. We claim that $\{A_{k_l}\}_{l\geq 1}$ is a F\o lner sequence for $G\ast_{\langle c=d \rangle} H^{\sigma}$. Indeed, $\{A_{k_l}\}$ is F\o lner for $G$, and for all $h\in H$, we have

\begin{eqnarray}
\lim_{l\rightarrow \infty}\frac{|A_{k_l} \vartriangle h\cdot A_{k_l}|}{|A_{k_l}|}&=&\lim_{l\rightarrow \infty}\frac{|A_{k_l} \vartriangle \sigma^{-1}h\sigma A_{k_l}|}{|A_{k_l}|}= \lim_{l\rightarrow \infty}\frac{|\sigma A_{k_l} \vartriangle h\sigma A_{k_l}|}{|A_{k_l}|}\nonumber \\
&=& \lim_{l\rightarrow \infty}\frac{|B_{k_l} \vartriangle hB_{k_l}|}{|B_{k_l}|}=0,\nonumber
\end{eqnarray}
since $\{B_{k_l}\}$ is F\o lner for $H$, $\sigma(A_{k_l})=B_{k_l}$ and $|A_{k_l}|=|B_{k_l}|$, for all $l\geq 1$.

Furthermore, if $H$ is a finite index subgroup of $\mathbb{F}_{n+1} \ast_{\langle c=d \rangle} \mathbb{F}_{m+1}$, since every finite index subgroup of $\mathbb{F}_{n+1}$ acts transitively on $X$, \textit{a fortiori} the $H$-action on $X$ is transitive.

Therefore, we have:

\begin{thm}\label{amalgam}
\begin{enumerate}
\item There exists a transitive, faithful and amenable action of the group $\mathbb{F}_{n+1} \ast_{\langle c=d \rangle} \mathbb{F}_{m+1}$ on $X$, where $c\in \mathbb{F}_{n+1}$ (respectively $d\in \mathbb{F}_{m+1}$) is a cyclically reduced non-primitive word such that the exponent sum of some generator occurring in $c$ (respectively $d$) is zero.
\item Every finite index subgroup of such a group admits transitive, faithful and amenable action on $X$.
\end{enumerate}
\end{thm}

The complete proof of Theorem \ref{CycPinch} is achieved from the following Lemma:

\begin{lem}
If $c$ is a reduced word in $\mathbb{F}_n$, then there exists an automorphism $\phi$ of $\mathbb{F}_n$ such that the exponent sum of some generator occurring in $\phi(c)$ is zero.
\end{lem}

\begin{proof}
Since there is an epimorphism $\pi : \textrm{Aut}(\mathbb{F}_n) \twoheadrightarrow \textrm{Aut}(\mathbb{Z}^n)\simeq GL_n(\mathbb{Z})$, it is enough to find a matrix $M\in GL_n(\mathbb{Z})$ such that the exponent sum $S_{\phi(c)}(t)$ of exponents of some generator $t$ in the word $\phi(c)$ is zero, where $\phi\in $ Aut$(\mathbb{F}_n)$ is such that $\pi(\phi)=M\in GL_n(\mathbb{Z})$. Denote by $t_1$, $\dots$, $t_n$ the generators of $\mathbb{F}_n$ such that $S_c(t_i)\neq 0$, $\forall 1\leq i\leq n$.
Let $m:=lcm(S_c(t_1), S_c(t_2))$ be the least common multiple of $S_c(t_1)$ and $S_c(t_2)$. Then there exist $m_1$ and $m_2$ such that $m=m_1S_c(t_1)$ and $m=m_2S_c(t_2)$ so that $m_1S_c(t_1)-m_2S_c(t_2)=0$. Moreover, the greatest common divisor $gcd(m_1,m_2)$ of $m_1$ and $m_2$ is 1, so by Bézout's identity, there exist $a$ and $b$ such that $m_1a+m_2b=1$. So by letting $s:=b S_c(t_1)+ aS_c(t_2)$, the matrix
$$
\left( \begin{array}{cc} \begin{array}{cc} m_1 & -m_2 \\ b & a \end{array} & \textrm{{\Large 0}} \\ \textrm{{\Large 0}} & \begin{array}{ccc} 1& & \\ & \ddots & \\ & & 1 \end{array} \end{array}\right)
$$
is in $GL_n(\mathbb{Z})$ and it sends $(S_c(t_1)$, $S_c(t_2)$, $\dots$, $S_c(t_n))^t$ to $(0$, $s$, $\dots$, $S_c(t_n))^t$.
\end{proof}


\bibliographystyle{amsplain}
\bibliography{amalgambib}

\smallskip

\begin{quote}
Soyoung \textsc{Moon} \\
Institut de Math\'ematiques \\
Universit\'e de Bourgogne\\
UMR 5584 du CNRS\\
9 avenue Alain Savary - BP 47870\\
21078 Dijon cedex\\
France

\smallskip

E-mail: \url{soyoung.moon@u-bourgogne.fr}
\end{quote}
\end{document}